\documentclass[11pt]{amsart}
\usepackage{amsmath, amssymb, amsthm}
\usepackage{hyperref}
\usepackage{tikz}
\usepackage{subfigure}
\usepackage{ytableau}
\usepackage{array}
\usepackage{amsfonts}
\usepackage{mathtools}
\theoremstyle{plain}
\newtheorem{theorem}{Theorem}[section]
\newtheorem{lemma}[theorem]{Lemma}
\newtheorem{corollary}[theorem]{Corollary}
\newtheorem{proposition}[theorem]{Proposition}

\newtheorem{remark}[]{Remark}
\newtheorem{claim}[]{Claim}
\newtheorem{case}[]{Case}

\renewcommand{\Im}{\operatorname{Im}}

\numberwithin{equation}{section}
\begin{document}
	\title[Asymptotic formula for biregular overpartitions
	]
	{Asymptotic formula for biregular overpartitions} 
	\author{JAYANTA BARMAN}
	\address{JAYANTA BARMAN\\ Department of Mathematics \\
		Indian Institute of Technology Kharagpur \\
		Kharagpur-721302,  India.} 
	\email{b1999jayanta@gmail.com, b1999jayanta@kgpian.iitkgp.ac.in}
	
	
	\begin{abstract} 
		Let $\bar{p}_{s,t}(N)$ be the number of overpartitions of $N$ into parts that are not divisible by $s$ or $t$. In this paper, we obtain asymptotic formulas for $\bar{p}_{s,t}(N)$ for relatively prime integers $s>t\ge 10$ via the saddle-point method. Our results cover different ranges of the parameters $s$ and $t$ relative to $N$. As an application, we show that $\bar{p}_{s,t}(N)$ satisfies higher-order Turán inequalities for all sufficiently large $N$ by using a result of Griffin, Ono, Rolen, and Zagier.
	\end{abstract}
	
	\maketitle
	\section{Introduction}\label{section1}
	Integer partitions are fundamental objects in combinatorics, geometry, 
	mathematical physics, number theory, and representation theory \cite{james}. Among 
	these, the partition function $p(N)$ plays a central role. In their 
	celebrated work, Hardy and Ramanujan \cite{hardy} introduced the circle 
	method and established the celebrated asymptotic formula for $p(N)$. 
	Since these foundational works, the circle method has been widely employed to derive asymptotic formulas and Rademacher-type series \cite{rademacher} for various restricted partition functions (see, e.g., \cite{bridges, hagis1964, hagis1965, hagis1971, sills}).
	
	One important generalization of ordinary partitions is the notion of an overpartition. An overpartition of a positive integer $N$ is a partition of $N$ in 
	which the first occurrence of each distinct part may be overlined. For 
	example, the overpartitions of $4$ are 
	$4$, $\overline{4}$, $3+1$, $\overline{3}+1$, $3+\overline{1}$, $\overline{3}+\overline{1}$, $2+2$, $\overline{2}+2$, $2+1+1$, $\overline{2}+1+1$, $2+\overline{1}+1$, $\overline{2}+\overline{1}+1$, $1+1+1+1$, and $\overline{1}+1+1+1$. The concept of overpartitions was introduced by Corteel and Lovejoy \cite{corteel}. Let $\bar{p}(N)$ denote the number of overpartitions of $N$. Since their introduction, overpartitions have attracted considerable attention, and many of their arithmetic, combinatorial, and analytic properties have been investigated in
	\cite{corteel, engel, hirschhorn}. 
	The asymptotic behavior of $\bar{p}(N)$ is well known and is given by
	\begin{align}\label{eq-pN}
		\bar{p}(N)
		=
		\frac{1}{8N}\exp\left(\pi\sqrt{N}\right)
		\left(1+O\left(N^{-\frac{1}{2}}\right)\right).
	\end{align} 
	We next recall the notion of $s$-regular partitions. For an integer 
	$s>1$, a partition of $N$ is called $s$-regular if none of its parts is 
	divisible by $s$. Let $p_s(N)$ denote the number of $s$-regular 
	partitions of $N$. A classical theorem of Glaisher \cite{Glaisher1883} 
	states that $p_s(N)$ is equal to the number of partitions of $N$ in 
	which no part occurs more than $s-1$ times. Various arithmetic, multiplicative, and monotonicity properties of
	$p_s(N)$ have been extensively investigated in
	\cite{barman2024arithmetic, beckwith2016multiplicative, carlson, cui,
		penniston, singh2023proofs}.
	
	The notion of an $s$-regular partition extends naturally to overpartitions. An overpartition of $N$ is called $s$-regular if none of its parts is divisible by $s$. Let $\overline{p}_s(N)$ denote the number of 
	$s$-regular overpartitions of $N$. This function has also received 
	considerable attention, and several of its arithmetic and combinatorial 
	properties have been studied in 
	\cite{alanazi, barman2018congruences, ray, shen}.
	
	These partitions naturally lead to the notion of biregular partitions.
	More precisely, a partition of $N$ is called $(s,t)$-biregular if none
	of its parts is divisible by either $s$ or $t$. Let $p_{s,t}(N)$ denote the number of
	$(s,t)$-biregular partitions of $N$. Laughlin and Parsell \cite{Jmc}
	established a Hardy--Ramanujan--Rademacher-type infinite series for
	$p_{s,t}(N)$ when $s$ and $t$ are relatively prime square-free integers.
	
	A natural analogue of biregular partitions in the setting of
	overpartitions is obtained by imposing the same divisibility
	restrictions on the parts. Namely, an overpartition of $N$ is called
	$(s,t)$-biregular if none of its parts is divisible by either $s$ or
	$t$. Let $\overline{p}_{s,t}(N)$ denote the number of
	$(s,t)$-biregular overpartitions of $N$. The generating function for
	$\overline{p}_{s,t}(N)$ was studied by Nadji et al.~\cite{nadji}, who
	also investigated several arithmetic properties of these
	overpartitions. Subsequently, further arithmetic and combinatorial
	properties of $(s,t)$-biregular overpartitions were studied in
	\cite{alanazi2, anakha, ghoshal}.
	
	Motivated by these developments, we obtain asymptotic formulas for $\bar{p}_{s,t}(N)$ for relatively prime integers $s>t\ge 10$ via the saddle-point method. Our results cover different ranges of the parameters $s$ and $t$ relative to $N$.
	Although the circle method is remarkably powerful and has been
	successfully applied to a wide variety of partition functions, its
	implementation can involve substantial technical complexity.
	The saddle-point method provides an alternative approach that is
	particularly well suited to extracting coefficient asymptotics
	directly from generating functions. A general framework for this approach was developed by Meinardus \cite{meinardus} without relying on the transformation properties of the generating function under the modular group. This method has since been extended and applied to various partition functions \cite{murty}. Later, Debruyne and Tenenbaum \cite{debruyne} independently developed a saddle-point method and applied it to several partition functions. Recently, Tyler \cite{tyler} introduced a saddle-point approach to establish asymptotic formulas for $t$-core partitions with explicit error terms, which can be used to derive asymptotic formulas for a broad class of restricted partition functions. Related
	applications of the saddle-point method can also be found in
	\cite{barman1, barman2026asymptotic, barman2}.
	Although our approach extends the framework of Tyler's method, the present problem requires a separate analysis:
	the generating function for $\overline{p}_{s,t}(N)$ and the resulting saddle-point
	equation differ from those arising in the $t$-core setting considered
	by Tyler. As an application, we prove
	that, for fixed $s$ and $t$ with $\gcd(s,t)=1$, the sequence
	$\{\overline{p}_{s,t}(N)\}_{N\geq 0}$ satisfies higher-order Turán
	inequalities for all sufficiently large $N$.  We therefore begin by recalling the necessary
	definitions and notation.

	
	Let $q=\exp(2\pi iz)$, $z=x+iy$ and $y>0$. Recall that the Dedekind eta function $\eta(z)$ is defined by
	\begin{equation*}
		\eta(z)=\exp\left(\frac{\pi iz}{12}\right)\prod_{n=1}^{\infty}(1-\exp(2\pi inz)).   
	\end{equation*}
	The generating function for $\overline{p}_{s,t}(N)$ \cite{nadji} is given by
	\begin{align}\label{eq-Hzst}
		\notag
		H(q,s,t)=\sum_{N=0}^{\infty}\bar{p}_{s,t}(N)q^{N}&=\prod_{n=1}^{\infty}\frac{\left(1+q^{n}\right)\left(1+q^{stn}\right)\left(1-q^{sn}\right)\left(1-q^{tn}\right)}{\left(1-q^{n}\right)\left(1-q^{stn}\right)\left(1+q^{sn}\right)\left(1+q^{tn}\right)}\\
		&=\frac{\eta(2z)\eta(2stz)\eta(sz)^{2}\eta(tz)^{2}}{\eta(z)^{2}\eta(stz)^{2}\eta(2sz)\eta(2tz)},
	\end{align}
	as
	\begin{align*}
		\prod_{n=1}^{\infty}\frac{1+q^{n}}{1-q^{n}}
		&= \prod_{n=1}^{\infty}\frac{1-q^{2n}}{(1-q^{n})^{2}}
		=\frac{\eta(2z)}{\eta(z)^{2}}.
	\end{align*}
	By Cauchy's integral formula, we have
	\begin{align*}
		\bar{p}_{s,t}(N)=\frac{1}{2\pi i}\int_{C}\frac{H(q,s,t)}{q^{N+1}}dq,  
	\end{align*}
	where $C$ is a simple positively oriented loop around the origin, located entirely within the unit circle. For a fixed value of $y$, as $x$ varies over any interval of length $1$, $q$ moves along a full circle of radius $e^{-2\pi y}$. Therefore, we can write $\bar{p}_{s,t}(N)$ as
	\begin{align}\label{eq-integral}
		\bar{p}_{s,t}(N)&=\int_{-1/2}^{1/2} \exp\left( -2\pi izN\right)H(z,s,t)dx.
	\end{align} 
	We use the functions $\mu_k$ ($k\geq1$) from \cite{tyler}, defined by
	\begin{equation}\label{demu(z)}
		\mu_{k}(z)=-\frac{z^{k+1}}{2\pi i} \left(\frac{d}{dz}\right)^{k} \log\eta(z).
	\end{equation}
	Let
	\begin{align}\label{eq-bksty}
		\notag
		b_k(s,t,y)&= 4st\mu_k(iy)+4\mu_k(stiy)+t\mu_k(2siy)+s\mu_k(2tiy)\\
		& -st\mu_{k}(2iy)-\mu_k(2stiy)-4t\mu_k(siy)-4s\mu_k(tiy).  
	\end{align}
	We are now ready to state our main theorems.
	\begin{theorem}\label{thm-1.1}
		Suppose that $N>0$ and $s>t\ge10$ are relatively prime integers. 
		\newline
		(i) There exists a unique solution $y>0$ to the equation
		\begin{align}\label{eq-solmain}
			&\frac{-b_1(s,t,y)}{2sty^{2}}=N.
		\end{align}
		(ii) Corresponding to this value of $y$, $\bar{p}_{s,t}(N)$ hold the following approximation:
		\begin{align*}
			\bar{p}_{s,t}(N)= \frac{y^{\frac{3}{2}}\sqrt{2st}\,\exp\left(2\pi Ny\right)H(iy,s,t)}{\sqrt{b_2(s,t,y)}}
			\left(1+O\left(y\right)\right).     
		\end{align*}
	\end{theorem}
	While the above formula is general, its direct application is inconvenient because the parameter $y$ is defined implicitly. In the following theorem, we apply Theorem~\ref{thm-1.1}$(i)$ to determine $y$ explicitly in the regime $2sty < 1$, thereby enabling us to derive a concrete asymptotic formula for $\bar{p}_{s,t}(N)$ from Theorem~\ref{thm-1.1}$(ii)$.
	\begin{theorem}\label{thm-1.2}
		Suppose that $s>t\ge 10$ are relatively prime integers. Then, for any fixed $\epsilon>0$ with $2st \le \frac{2\pi}{\left(\frac{1}{2} + \epsilon\right) \log N} \sqrt{\frac{16stN}{(s-1)(t-1)}}$, we have
		\begin{align*}
			\bar{p}_{s,t}(N)=\frac{1}{2\sqrt{2}\,N^{\frac{3}{4}}}\left(\frac{(s-1)(t-1)}{st}\right)^{\frac{1}{4}}\exp\left(\pi\sqrt{\frac{N(s-1)(t-1)}{st}}\right)\left(1+O\left(N^{-\frac{1}{2}}\right)\right).   
		\end{align*}
	\end{theorem}
	The next theorem follows by applying Theorem~\ref{thm-1.1} under the condition $ty\geq 1$.
	\begin{theorem}\label{thm-1.3}
		Let $s > t > 4\sqrt{N}$ be coprime integers, and let $1 < \delta_1, \delta_2, \delta_3, \delta_4 < 1.01$. Then
		\begin{align*}
			\bar{p}_{s,t}(N)=&\bar{p}(N)E_3(N,s,t)\left(1+O\left(N^{-\frac{1}{2}}\right)\right),
		\end{align*}
		where
		\begin{align*}
			E_3(N,s,t)=&\exp\left(\delta_1\exp\left(-\frac{\pi s}{\sqrt{N}}\left(1+O\left(N^{-\frac{1}{2}}\right)\right)\right)-2\delta_2\exp\left(-\frac{\pi s}{2\sqrt{N}}\left(1+O\left(N^{-\frac{1}{2}}\right)\right)\right)\right)\\
			\notag
			\times & \exp\left(\delta_3\exp\left(-\frac{\pi t}{\sqrt{N}}\left(1+O\left(N^{-\frac{1}{2}}\right)\right)\right)-2\delta_4\exp\left(-\frac{\pi t}{2\sqrt{N}}\left(1+O\left(N^{-\frac{1}{2}}\right)\right)\right)\right).   
		\end{align*}
		
	\end{theorem}
	The next theorem follows by applying Theorem~\ref{thm-1.1} under the condition $2ty< 1$ and $sy\ge 1$.
	\begin{theorem}\label{thm-1.4}
		Let $s>4\sqrt{N}$ and
		$2t\le \frac{8\pi}{\left(\frac{1}{2}+\epsilon_1\right)\log N}\sqrt{\frac{tN}{t-1}}$ be coprime integers, where $\epsilon_1>0$ is fixed,
		and let $1<\delta_7,\delta_8,\delta_9,\delta_{10}<1.01$. Then
		\begin{align*}
			\bar{p}_{s,t}(N)=\frac{1}{2\sqrt{2t}\,N^{\frac{3}{4}}}\left(1-\frac{1}{t}\right)^{\frac{1}{4}}\exp\left(\pi \sqrt{N\left(1-\frac{1}{t}\right)}\right)E_5(N,s,t)\left(1+O\left(N^{-\frac{1}{2}}\right)\right), 
		\end{align*}
		where 
		\begin{align*}
			&E_5(N,s,t)\\
			&=\exp\left(2\delta_7\exp\left(-\frac{\pi s}{2}\sqrt{\frac{t(t-1)}{N}}\left(1+O\left(N^{-\frac{1}{2}}\right)\right)\right)-\delta_8\exp\left(-\pi s\sqrt{\frac{t(t-1)}{N}}\left(1+O\left(N^{-\frac{1}{2}}\right)\right)\right)\right)\\
			&\times \exp\left(\delta_9\exp\left(-\pi s\sqrt{\frac{(t-1)}{tN}}\left(1+O\left(N^{-\frac{1}{2}}\right)\right)\right)-2\delta_{10}\exp\left(-\frac{\pi s}{2}\sqrt{\frac{(t-1)}{tN}}\left(1+O\left(N^{-\frac{1}{2}}\right)\right)\right)\right).
		\end{align*}
	\end{theorem}
	\begin{remark}
		Throughout this paper, we restrict our attention to the three cases considered above; the remaining cases can be handled similarly.
	\end{remark}
	
	If both $s$ and $t$ exceed $N$, then Theorem~\ref{thm-1.3} reduces to
	\eqref{eq-pN}. Similarly, for $s>N$, Theorem~\ref{thm-1.4} implies the following corollary, which gives an asymptotic formula for the number of $t$-regular overpartitions of $N$. This agrees with Corollary~1.2 of \cite{dey}.
	\begin{corollary}
		Let $s$ and $t$ be relatively prime integers with $s>N$, and suppose that $t$ is fixed. Then
		\begin{align*}
			\bar{p}_{s,t}(N)=\frac{1}{2\sqrt{2t}\,N^{\frac{3}{4}}}\left(1-\frac{1}{t}\right)^{\frac{1}{4}}\exp\left(\pi \sqrt{N\left(1-\frac{1}{t}\right)}\right)\left(1+O\left(N^{-\frac{1}{2}}\right)\right).
		\end{align*}  
	\end{corollary}
	
	In the following subsection, we discuss log-concavity and higher-order Turán inequalities.
	\subsection{Higher-Order Turán Inequalities for \texorpdfstring{$\bar{p}_{s,t}(N)$}{pbar(s,t)(N)}}
	A real sequence $\{a(N)\}$ is log-concave at $N$ if
	$$
	a(N)^2 \ge a({N-1})a({N+1}), \quad\text{for all }\quad N\ge 1.
	$$ 
	This is equivalent to saying that the
	quadratic polynomial
	\begin{equation*}
		J^{2,N-1}_{a}(V) = a({N-1}) + 2a(N) V + a({N+1}) V^2
	\end{equation*}
	has only real roots. More generally, the Jensen polynomial of
	degree $d$ and shift $N$ is defined by
	\begin{equation*}
		J^{d,N}_{a}(V): = \sum_{j=0}^{d} \binom{d}{j} a({N+j}) V^j.
	\end{equation*}
	The sequence $\{a(N)\}$ satisfies the Tur\'an inequality of order
	$d$ at $N$ if and only if $J^{d,N-1}_{a}(V)$ is hyperbolic (has only real
	roots). Griffin, Ono, Rolen, and Zagier \cite{GORZ19} proved that for
	broad classes of sequences—including the partition function $p(N)$ and
	Fourier coefficients of weakly holomorphic modular forms—the normalized
	Jensen polynomials $J^{d,N-1}_{a}(V)$ converge to the $d$-th Hermite
	polynomial $H_d(V)$ as $N \to \infty$. Consequently,
	$J^{d,N-1}_{a}(V)$ is hyperbolic for all sufficiently large $N$.
	
	We state their general criterion below.
	\begin{theorem}[{\cite[Theorems~3 and~8]{GORZ19}}]\label{thm_GORZ}
		Let $\{a(N)\}, \{L(N)\},$ and $ \{\delta(N)\}$ be sequences of positive real numbers and $\delta(N)\to0$ as $N\to\infty$.  
		For integers $r \geq 0$,  $d \geq 1$, suppose that there exist real numbers $h_3(N), h_4(N), \ldots, h_d(N)$, for which 
		\begin{align*}
			\log \left( \frac{a(N+r)}{a(N)} \right)
			= L(N)r - \delta(N)^2 r^2 + \sum_{j=3}^{d} h_j(N) r^j + o(\delta(N)^d),
		\end{align*}
		as $N\to\infty$, with $h_j(N) = o(\delta(N)^j)$ for each $3 \leq j \leq d$.  
		Then we have  
		\begin{align*}
			\lim_{N \to \infty} \left( \frac{\delta(N)^{-d}}{a(N)}
			J_{a}^{d,N} \!\left( \frac{\delta(N)V - 1}{\exp(L(N))} \right) \right)
			= H_d(V).
		\end{align*}
	\end{theorem}
	The Hermite polynomial $H_d(V)$ has simple real zeros. Since an affine
	change of variables preserves this property, the above theorem shows that
	the associated Jensen polynomials have distinct real zeros when $N$ is
	sufficiently large. Hence, the corresponding higher-order Tur\'{a}n
	inequalities hold for all sufficiently large $N$.
	Log-concavity and higher-order Tur\'{a}n inequalities have been studied
	for several partition functions; see, for example,
	\cite{agarwal,chen, desalvo, dey, Dimitrov, larson, ono,pandey}.
	In this direction, we establish the following result for the partition
	function $\bar{p}_{s,t}(N)$, which gives higher-order Tur\'{a}n inequalities for
	all sufficiently large $N$.
	\begin{theorem}\label{thm-turan}
		Let $s > t \ge 10$ be coprime integers. For each positive integer $d$, the Jensen polynomial $J^{d,N}_{\bar{p}_{s,t}(N)}(V)$ is hyperbolic for all sufficiently large $N$.
	\end{theorem}
	We now sketch the proof of the general asymptotic formula for $\bar{p}_{s,t}(N)$.
	\subsection{\texorpdfstring{Sketch of the proof of Theorem \ref{thm-1.1}}{}}
	Our saddle-point analysis relies on the integral in \eqref{eq-integral}. Since the right-hand side of \eqref{eq-integral} depends solely on $\Im z = y$, while the left-hand side is independent of $y$, we can choose $y$ optimally to derive the asymptotic formula for $\bar{p}_{s,t}(N)$. The Taylor expansion of $H(z,s,t)$ converges in the region $|x|<y$, so we split the integral in (\ref{eq-integral}) into the following two ranges:
	\begin{align*}
		|x|<\frac{y}{3}
		\qquad \text{and} \qquad
		\frac{y}{3}\le |x|\le \frac{1}{2}.
	\end{align*}
	Accordingly, equation (\ref{eq-integral}) can be rewrite as
	\begin{equation}\label{eq-qNt}
		\begin{aligned}
			\bar{p}_{s,t}(N)
			&=\exp(2\pi Ny)H(iy,s,t)\int_{-y/3}^{y/3}\exp\left(-2\pi i Nx+2\pi i\frac{1}{2\pi i}\log\frac{H(z,s,t)}{H(iy,s,t)}\right)dx\\
			&+\exp(2\pi Ny)H(iy,s,t)\int_{\frac{y}{3}\le|x|\le \frac{1}{2}}\exp\left(-2\pi i Nx\right)\frac{H(z,s,t)}{H(iy,s,t)}dx.
		\end{aligned}
	\end{equation}
	The Taylor expansion of $\log H(z,s,t)$ (see (\ref{eq-Hzst})) near $x=0$ is given by
	\begin{align*}
		\log H(z,s,t)&=\log H(iy,s,t)+x\frac{d}{dz} \log H(iy,s,t)+\frac{x^{2}}{2!}\frac{d^{2}}{dz^{2}}\log H(iy,s,t)+\cdots.
	\end{align*}
	From the definition of $\mu_k$, we have
	\begin{align*}
		\left(\frac{d}{dz}\right)^{k}\log H(z,s,t)
		&=\frac{2\pi i}{2stz^{k+1}}( 4st\mu_k(z)+4\mu_k(stz)+t\mu_k(2sz)+s\mu_k(2tz)\\
		& -st\mu_{k}(2z)-\mu_k(2stz)-4t\mu_k(sz)-4s\mu_k(tz)). 
	\end{align*}
	Therefore,
	\begin{equation}\label{eq-tay1}
		\begin{aligned}
			\frac{1}{2\pi i}\log\frac{H(z,s,t)}{H(iy,s,t)} 
			&=\sum_{k=1}^{\infty}\frac{x^{k}}{k!}\frac{b_k(s,t,y)}{2st(iy)^{k+1}}.
		\end{aligned}
	\end{equation}
	The Taylor expansion above is used to estimate the integral in
	\eqref{eq-qNt} over the range $|x|<\frac{y}{3}$. For the complementary range
	$\frac{y}{3}\le |x|\le \frac{1}{2}$, where the expansion is not valid, we establish an
	upper bound in Lemma~\ref{pro-error}. In Proposition~\ref{thm-main},
	we show that the contribution from the first region yields the main
	term, while the contribution from the second region is absorbed into
	the error term.
	
	The paper is organized as follows. Section~\ref{sec-3} contains preliminary results and auxiliary estimates. It also includes the proof of Lemma~\ref{pro-error}. In Section~\ref{sec-4}, we prove Proposition~\ref{thm-main}, establish Theorem~\ref{thm-1.1} via the saddle-point method, and then deduce Theorems~\ref{thm-1.2}, \ref{thm-1.3}, \ref{thm-1.4}, and \ref{thm-turan}.
	
	\section{Acknowledgments}
	The author thanks the University Grants Commission (UGC), India, for support through the Fellowship Programme. The author is also supported by the ARG-MATRICS Grant (Grant No. ANRF/ARGM/2025/002540/MTR) of Kamalakshya Mahatab.
	
	\section{Preliminaries}\label{sec-3}
	This section presents the auxiliary results required for our main proofs. We review several known facts from \cite{tyler} and related work. We then adapt existing bounds and prove several new results.
	
	The following lemma is useful for the case $y\geq 1$.
	\begin{lemma}[{\cite[Lemma 3.2]{tyler}}]\label{lemma-2.0}
		For every $y\ge \frac{\sqrt{3}}{2}$, there exists a constant $\rho$ satisfying $|\rho|<1$ such that
		\begin{align*}
			|\eta(z)|=\exp\left(-\frac{\pi y}{12}+\rho 1.01e^{-2\pi y}\right).
		\end{align*}
	\end{lemma}
	We require the following estimate for $\eta(iy)$ when $y\ge 1$.
	\begin{lemma}[{\cite[Lemma 3.3]{barman1}}]\label{lemma-2.1}
		For each $y\ge \frac{\sqrt{3}}{2}$, there exists a constant $\delta$ with $1<\delta<1.00873$ such that
		\begin{equation*}
			\eta(iy) = \exp\left(-\frac{\pi y}{12} - \delta e^{-2\pi y}\right).
		\end{equation*}
	\end{lemma}
	The next lemma is useful for small $y > 0$.
	\begin{lemma}[{\cite[Lemma 3.4]{barman1}}]\label{lemma-2.2}
		Corresponding to each $y\in(0,1)$, there exists a constant $\nu\in(1,1.00873)$ such that
		\begin{equation*}
			\eta(iy) = y^{-\frac{1}{2}} \exp\left( -\frac{\pi}{12y} - \nu e^{-\frac{2\pi}{y}} \right).
		\end{equation*}
	\end{lemma}
	We now derive an estimate for the integral in \eqref{eq-qNt} over the range $\frac{y}{3} \le |x| \le \frac{1}{2}$.
	\begin{lemma}\label{pro-error}
		If $y\le\frac{1}{1000}$ and $s>t\ge 50$ are relatively prime integers, then  
		\begin{align*}
			\int_{\frac{y}{3}\le |x|\le \frac{1}{2}}\left|\frac{H(z,s,t)}{H(iy,s,t)}\right|dx\le 1.5\exp\left(-\frac{\pi }{600y}\right).
		\end{align*}
	\end{lemma}
	\begin{proof}
		From the definition of $H(z,s,t)$, it follows that
		\begin{align*}
			\left|\frac{H(z,s,t)}{H(iy,s,t)}\right| =\left|\frac{\eta(iy)^{2}}{\eta(z)^{2}}\right|
			\left|\frac{\eta(2z)}{\eta(2iy)}\right|\left|\frac{\eta(stiy)^{2}}{\eta(stz)^{2}}\right|\left|\frac{\eta(2stz)}{\eta(2stiy)}\right|\left|\frac{\eta(sz)^{2}}{\eta(siy)^{2}}\right|
			\left|\frac{\eta(2siy)}{\eta(2sz)}\right|\left|\frac{\eta(tz)^{2}}{\eta(tiy)^{2}}\right|\left|\frac{\eta(2tiy)}{\eta(2tz)}\right|.
		\end{align*}
		Let $F(z) = \frac{\eta(2z)}{\eta(z)}$. We first consider the quantity
		\begin{align}\label{eq-atzaty}
			|G(z)|= \left|\frac{\eta(iy)^{2}}{\eta(z)^{2}}\right|
			\left|\frac{\eta(2z)}{\eta(2iy)}\right|=\left|\frac{F(z)}{F(iy)}\right|
			\left|\frac{\eta(iy)}{\eta(z)}\right|.
		\end{align}
		From the proof of Lemma~3.2 in \cite{barman2026asymptotic}, we have
		\begin{align}\label{eq-error}
			\left|\frac{\eta(iy)}{\eta(z)}\right|\le 1.2\exp\left(-\frac{\pi}{120y}\right). 
		\end{align}
		Using the definition of $\eta(z)$, we obtain
		\begin{align*}
			F(z)=q^{\frac{1}{24}}\prod_{n=1}^{\infty}(1+q^{n}).  
		\end{align*}
		Let $r_n=\exp(-2\pi ny)$. Then
		\begin{align}\label{eq-rn}
			\left|\frac{F(z)}{F(iy)}\right|=\prod_{n=1}^{\infty}\frac{\left|1+r_n\exp(2\pi inx)\right|}{1+r_n}. 
		\end{align}
		For $0\le\tilde{r}<1$, we have
		\begin{align}\label{eq-r}
			\frac{\left|1+\tilde{r}e^{i\theta}\right|^{2}}{(1+\tilde{r})^{2}}=1-\frac{2\tilde{r}(1-\cos \theta)}{(1+\tilde{r})^{2}}.
		\end{align}
		From (\ref{eq-rn}), we deduce
		\begin{equation} \label{eq-loggg}
			\begin{aligned}
				\log\frac{|F(z)|}{|F(iy)|}
				&=\sum_{n=1}^{\infty}\frac{1}{2}\log \left(\frac{|1+r_n\exp(2\pi inx)|^{2}}{(1+r_n)^{2}}\right)\\
				&=\sum_{n=1}^{\infty}\frac{1}{2}\log\left(1-\frac{2r_n(1-\cos(2\pi  nx))}{(1+r_n)^{2}}\right)\quad(\text{using (\ref{eq-r})})\\
				&\le -\sum_{n=1}^{\infty}\frac{r_n(1-\cos(2\pi  nx))}{(1+r_n)^{2}}\quad(\text{using $\log(1-u)\le -u$})\\
				&\le -\frac{r_{N_{0}}}{4}\sum_{n=1}^{N_0}(1-\cos(2\pi nx)).
			\end{aligned}
		\end{equation}
		Now we choose $N_0=\left\lfloor \frac{1}{y} \right\rfloor$, the term $r_{N_0}=\exp(-2\pi N_0y)\ge \exp(-2\pi)$ is bounded below by an absolute constant. Therefore, it is obvious that 
		\begin{align*}
			\sum_{n=1}^{N_0}(1-\cos(2\pi nx))\ge 0
		\end{align*}
		uniformly for $\frac{y}{3}\le x\le \frac{1}{2}.$
		\newline
		Substituting the above estimate in (\ref{eq-loggg}), we obtain the following:
		\begin{align*}
			\left|\frac{F(z)}{F(iy)}\right|\le    1.
		\end{align*}
		Combining this with (\ref{eq-error}) in (\ref{eq-atzaty}), we get
		\begin{align}\label{eq-Err1}
			\left|G(z)\right|\le 1.2\exp\left(-\frac{\pi}{120y}\right).  
		\end{align}
		\begin{case}
			\textbf{\boldmath $2sty<1$.} 
		\end{case}
		Using a similar approximation to that employed in the estimate of $\left|G(z)\right|$, we obtain 
		\begin{align}\label{eq-Err2}
			\left|\frac{\eta(stiy)^{2}}{\eta(stz)^{2}}\right|\left|\frac{\eta(2stz)}{\eta(2stiy)}\right|\le 1.2 \exp\left(-\frac{\pi}{120sty}\right).
		\end{align}
		Employing an estimate analogous to \eqref{eq-error}, we derive
		\begin{align*}
			\left|\frac{\eta(2siy)}{\eta(2sz)}\right| \le 1.2\exp\left(-\frac{\pi}{240sy}\right)\quad \text{and}\quad  \left|\frac{\eta(2tiy)}{\eta(2tz)}\right|  \le 1.2 \exp\left(-\frac{\pi}{240ty}\right). 
		\end{align*}
		Inserting Lemma~\ref{lemma-2.2} and Lemma~3.3 of \cite{tyler} with $1<\nu_1,\nu_2<1.01$, we obtain
		\begin{align*}
			&\left|\frac{\eta(sz)^{2}}{\eta(siy)^{2}}\right|\le \frac{49}{81}(sy)^{\frac{1}{2}}\exp\left(\frac{\pi}{6sy}+2\nu_1e^{-\frac{2\pi}{sy}}\right)\quad\text{and}\quad\\  &\left|\frac{\eta(tz)^{2}}{\eta(tiy)^{2}}\right|\le \frac{49}{81}(ty)^{\frac{1}{2}}\exp\left(\frac{\pi}{6ty}+2\nu_2e^{-\frac{2\pi}{ty}}\right).  
		\end{align*}
		Applying the above approximations and simplifying for $2sty<1$, we deduce that
		\begin{align}\label{eq-Err3}
			&\left|\frac{\eta(2siy)}{\eta(2sz)}\right| \left|\frac{\eta(sz)^{2}}{\eta(siy)^{2}}\right|\le \exp\left(\frac{\pi}{6sy}\right)\quad\text{and}\\
			\label{eq-Err4}
			&\left|\frac{\eta(2tiy)}{\eta(2tz)}\right|\left|\frac{\eta(tz)^{2}}{\eta(tiy)^{2}}\right|\le \exp\left(\frac{\pi}{6ty}\right).
		\end{align}
		Combining \eqref{eq-Err1}, \eqref{eq-Err2}, \eqref{eq-Err3}, and \eqref{eq-Err4}, we obtain
		\begin{align*}
			\left|\frac{H(z,s,t)}{H(iy,s,t)}\right|&\le 1.5\exp\left(-\frac{\pi }{120y}+\frac{\pi }{6sy}+\frac{\pi}{6ty}\right)\\
			&\le 1.5\exp\left(-\frac{\pi }{600y}\right),\quad\text{as $s>t\ge 50$}.
		\end{align*}
		\begin{case}
			\textbf{\boldmath $ty\ge 1$.} 
		\end{case}
		Substituting Lemma~\ref{lemma-2.0} and simplifying for $ty \ge 1$, we obtain
		\begin{align*}
			\left|\frac{\eta(stiy)^{2}}{\eta(stz)^{2}}\right|\left|\frac{\eta(2stz)}{\eta(2stiy)}\right|\left|\frac{\eta(sz)^{2}}{\eta(siy)^{2}}\right|
			\left|\frac{\eta(2siy)}{\eta(2sz)}\right|\left|\frac{\eta(tz)^{2}}{\eta(tiy)^{2}}\right|\left|\frac{\eta(2tiy)}{\eta(2tz)}\right| \le\exp\left(18.18 e^{-2\pi}\right).  
		\end{align*}
		It follows from \eqref{eq-Err1} and the above estimate that
		\begin{align*}
			\left|\frac{H(z,s,t)}{H(iy,s,t)}\right|\le 1.5\exp\left(-\frac{\pi }{600y}\right). 
		\end{align*}
		\begin{case}
			\textbf{\boldmath $2ty<1$ and $sy\ge 1$.} 
		\end{case}
		A similar application of Lemma~\ref{lemma-2.0} yields
		\begin{align*}
			\left|\frac{\eta(stiy)^{2}}{\eta(stz)^{2}}\right|\left|\frac{\eta(2stz)}{\eta(2stiy)}\right|\left|\frac{\eta(sz)^{2}}{\eta(siy)^{2}}\right|
			\left|\frac{\eta(2siy)}{\eta(2sz)}\right| \le\exp\left(12.12 e^{-2\pi}\right).   
		\end{align*}
		Combining \eqref{eq-Err1} and \eqref{eq-Err4} with the above estimates, we derive
		\begin{align*}
			\left|\frac{H(z,s,t)}{H(iy,s,t)}\right|\le 1.5\exp\left(-\frac{\pi }{600y}\right). 
		\end{align*}
		This concludes the proof.
	\end{proof}
	
	In \cite{tyler}, Tyler derived the following two expansions for $\mu_k(z)$ (see equations (4.14) and (4.16)). These formulas are particularly useful depending on whether $\Im z$ is large\footnote{Throughout this article, we describe the imaginary part as \emph{large} when it belongs to $[1,\infty)$, and as \emph{small} when it lies in $(0,1)$.} or small. Brief proofs of the following two lemmas were presented in \cite[Propositions~3.5 and ~3.6]{barman2}.
	\begin{lemma}\label{eq-(2.1)}
		For a large imaginary part, we use the following formula:
		\begin{align*}
			\notag
			\mu_{k}(z)&=\sum_{n=1}^{\infty}z^{k+1}(2\pi in)^{k-1}\sigma(n)\exp(2\pi inz)- \begin{cases} 
				\frac{z^{2}}{24} & \mbox{if } k=0,1 \\ 
				0 & \mbox{if } k\ge 2 
			\end{cases} \quad\text{and}\\  
			\mu_{k}^{\prime}(z)&=\sum_{n=1}^{\infty}\left((2\pi inz)^{k+1}+(k+1)(2\pi inz)^{k}\right) \frac{\sigma(n)}{2\pi in}\exp(2\pi inz)-\begin{cases} 
				\frac{z}{12} & \mbox{if } k=0,1 \\ 
				0 & \mbox{if } k\ge 2. 
			\end{cases}\\ 
			\notag   
		\end{align*}
	\end{lemma}
	\begin{lemma}\label{eq-2.12}
		When the imaginary part is small, we use the following formula:
		\begin{align*}
			\mu_{k}(z)&=\sum_{n=1}^{\infty}P_{k}\left(\frac{2\pi in}{z}\right)\sigma(n)\exp\left(-\frac{2\pi in}{z}\right)+\frac{(-1)^{k}k!}{24}+\frac{z}{4\pi i}\begin{cases} 
				\log(-iz) & \text{if } k=0 \\ 
				(-1)^{k-1}(k-1)! & \mbox{if } k\ge 1 
			\end{cases}\\
			\notag
			\text{and}\quad \mu_{k}^{\prime}(z)&=\sum_{n=1}^{\infty}Q_{k}\left(\frac{2\pi in}{z}\right)\frac{\sigma(n)}{2\pi in}\exp\left(-\frac{2\pi in}{z}\right)+\frac{1}{4\pi i}\begin{cases} 
				1+\log (-iz) & \mbox{if } k=0 \\ 
				(-1)^{k-1}(k-1)! & \mbox{if } k\ge 1, 
			\end{cases}          
		\end{align*}  
		where $P_0(w)=w^{-1}$ and $P_{k}(w)=(w-k)P_{k-1}(w)-wP^{\prime}_{k-1}(w)$ and $Q_k(w)=w^{2}(P_k(w)-P^{\prime}_k(w))$. For $k=0,1,2,3,4$, explicit values of $P_k$ and $Q_k$ are given in \cite{tyler} (see (4.18)). 
	\end{lemma}
	\begin{remark}
		For the next three lemmas, we do not present the full details of the approximations of the infinite sums. The main contribution comes from the initial terms, while the remaining terms can be estimated using standard methods. To keep the exposition concise and focused on the main ideas, we omit these routine calculations.
	\end{remark}
	Next, we establish three lemmas, which play a key role in the proof of Proposition \ref{thm-main}. For convenience, we recall the notation
	\begin{align}\label{eq-bksty1}
		\notag
		b_k(s,t,y)&= 4st\mu_k(iy)+4\mu_k(stiy)+t\mu_k(2siy)+s\mu_k(2tiy)\\
		& -st\mu_{k}(2iy)-\mu_k(2stiy)-4t\mu_k(siy)-4s\mu_k(tiy). 
	\end{align} 
	from \eqref{eq-bksty}, which will be used throughout the remainder of the paper.
	
	\begin{lemma}\label{lemma-mu2}
		Let $0 < y \le \frac{1}{10}$, and let $s > t \ge 10$ be coprime integers.\\
		$(i)$ If $2sty < 1$, then
		\begin{align*}
			\frac{0.99(s-1)(t-1)}{4}<b_2(s,t,y)<\frac{(s+1)(t+1)}{4}.
		\end{align*}
		$(ii)$ If $ty \ge 1$, then
		\begin{align*}
			\frac{0.9st}{4}<b_2(s,t,y)<\frac{(s+1)(t+1)}{4}.
		\end{align*}
		$(ii)$ If $sy \ge 1$ and $2ty < 1$, then
		\begin{align*}
			\frac{0.9s(t-1)}{4}<b_2(s,t,y)<\frac{(s+1)(t+1)}{4}.
		\end{align*}
	\end{lemma}
	\begin{proof}
		By Lemma~4.2 of \cite{tyler}, we have 
		\begin{align}\label{eq-mu20}
			0 < \mu_2(iy) < \frac{1}{12}.  
		\end{align} 
		Hence, it immediately follows that
		\begin{align*}
			b_2(s,t,y) < \frac{4st + 4 + s + t}{12}<\frac{(s+1)(t+1)}{4}.    
		\end{align*}
		For $0<my<1$, Lemma~\ref{eq-2.12} yields
		\begin{align}\label{eq-mu2210}
			\mu_{2}(miy) &= \sum_{n=1}^{\infty}\left(\frac{2\pi n}{my}-2\right)\sigma(n) \exp\left(-\frac{2\pi n}{my}\right) + \frac{1}{12} - \frac{my}{4\pi}.
		\end{align}
		While for $my\ge 1$ gives 
		\begin{align}\label{eq-mu2211}
			\mu_2(miy)=\sum_{n=1}^{\infty}(my)^{3}(2\pi n)\sigma(n)\exp(-2\pi mny).
		\end{align}
		\item[$(i)$] Consider the regime $2sty < 1$. For $s > t \ge 10$, the infinite series corresponding to $t\mu_2(2siy)$ and $s\mu_2(2tiy)$ dominate those of $4t\mu_2(siy)$ and $4s\mu_2(tiy)$, yielding
		\begin{align*}
			b_2(s,t,y) \ge \frac{(s-1)(t-1)}{4} - A_1.
		\end{align*}
		Using $y \le 1/10$, $2sty < 1$, and the standard bound $\sigma(n) < n^2$, we estimate the residual term $A_1$ by
		\begin{align*}
			A_1 &= st \sum_{n=1}^{\infty} \left( \frac{\pi n}{y} - 2 \right) \sigma(n) \exp\left(-\frac{\pi n}{y}\right) + \sum_{n=1}^{\infty} \left( \frac{\pi n}{sty} - 2 \right) \sigma(n) \exp\left(-\frac{\pi n}{sty}\right) \\
			&\le st \sum_{n=1}^{\infty} (10\pi n - 2) n^2 e^{-10\pi n} + \sum_{n=1}^{\infty} (2\pi n - 2) n^2 e^{-2\pi n}.
		\end{align*}
		Applying this upper bound on $A_1$, we arrive at
		\begin{align*}
			b_2(s,t,y) \ge \frac{0.99(s-1)(t-1)}{4}.
		\end{align*}
		$(ii)$  Next, we consider the case $ty \geq 1$. Applying the estimates from
		\eqref{eq-mu2210} to $\mu_2(iy)$ and $\mu_2(2iy)$, and from
		\eqref{eq-mu2211} to the remaining terms, and estimating the resulting
		exponential tails, we obtain
		\begin{align*}
			b_2(s,t,y) > \frac{3st}{12}-\frac{2sty}{4\pi}-A_2-A_3>\frac{0.9st}{4},
		\end{align*}
		where
		\begin{align*} 
			&A_2 = st \sum_{n=1}^{\infty} \left( \frac{\pi n}{y} - 2 \right) \sigma(n) \exp\left(-\frac{\pi n}{y}\right)+\sum_{n=1}^{\infty}(2sty)^{3}(2\pi n)\sigma(n)\exp(-4\pi nsty),\\ 
			&A_3=4t\sum_{n=1}^{\infty}(sy)^{3}(2\pi n)\sigma(n)\exp(-2\pi sny)+4s\sum_{n=1}^{\infty}(ty)^{3}(2\pi n)\sigma(n)\exp(-2\pi tny). 
		\end{align*}
		Here, the bounds for $A_2$ and $A_3$ follow from
		$sy>ty\geq1$, $0<y\leq1/10$, and $s>t\geq10$.
		\newline
		\item[$(iii)$] Finally, consider $sy \ge 1$ and $2ty < 1$. Combining \eqref{eq-mu2210} for $\mu_2(iy), \mu_2(2iy), \mu_2(tiy), \mu_2(2tiy)$ with \eqref{eq-mu2211} for $\mu_2(siy), \mu_2(2siy), \mu_2(stiy), \mu_2(2stiy)$, we obtain
		\begin{align*}
			b_2(s,t,y) \ge \frac{s(t-1)}{4} - A_2 - A_4 \ge \frac{0.9s(t-1)}{4},
		\end{align*}
		where $A_2$ is defined as above and
		\begin{align*}
			A_4 = 4t \sum_{n=1}^{\infty} (sy)^3 (2\pi n) \sigma(n) \exp({-2\pi sny}) + 4s \sum_{n=1}^{\infty} \left( \frac{2\pi n}{ty} - 2 \right) \sigma(n) \exp\left({-\frac{2\pi n}{ty}}\right).
		\end{align*}
		Here, the bound for $A_4$ follows by estimating the corresponding
		exponential tails using $sy\geq1$, $2ty<1$, $0<y\leq1/10$, and
		$s>t\geq10$. 
		
	\end{proof}
	\begin{lemma}\label{lemma-mu3}
		If $0<y\le \frac{1}{10}$ and $s>t\ge 10$ are relatively prime integers, then
		\begin{align*}
			\left|\frac{b_3(s,t,y)}{b_2(s,t,y)}\right|<6.  
		\end{align*}
	\end{lemma}
	\begin{proof}
		By Lemma~4.3 of \cite{tyler}, we have $-\frac{1}{4} < \mu_3(iy) < 0$ for all $y > 0$. Since $s>t\ge10$, all terms in each pair share the same sign. Using the fact that $|a-b| \le \max\{|a|,|b|\}$ for any real numbers $a, b$ of the same sign, we obtain
		\begin{align*}
			|b_3(s,t,y)| &\le \max\bigl\{|4st\mu_3(iy)|,\, |st\mu_3(2iy)|\bigr\} + \max\bigl\{|4\mu_3(stiy)|,|\mu_3(2stiy)|\bigr\} \\
			&\quad + \max\bigl\{|t\mu_3(2siy)|,\, |4t\mu_3(siy)|\bigr\} + \max\bigl\{|s\mu_3(2tiy)|,\, |4s\mu_3(tiy)|\bigr\} \\
			&< \frac{4st + 4 + 4t + 4s}{4} 
			= (s+1)(t+1).
		\end{align*}
		Combining the preceding estimate with Lemma~\ref{lemma-mu2}, we conclude that
		\begin{align*}
			\left|\frac{b_3(s,t,y)}{b_2(s,t,y)}\right|<6.  
		\end{align*}       
	\end{proof}
	\begin{lemma}\label{lemma-mu4}
		Suppose that $z=x+iy$ with $0<y\le \frac{1}{10}$ and $|x|<\frac{y}{3}$. For any $s>t\ge 10$ that are relatively prime integers, we have
		\begin{align*}
			\left|\frac{\ell(s,t,z)}{b_2(s,t,y)}\right|<36,
		\end{align*}    
		where
		\begin{align*}
			\ell(s,t,z)&=4st\mu_4(z)+4\mu_4(stz)+t\mu_4(2sz)+s\mu_4(2tz)\\
			& -st\mu_{4}(2z)-\mu_4(2stz)-4t\mu_4(sz)-4s\mu_4(tz).  
		\end{align*}
	\end{lemma} 
	\begin{proof}
		According to Lemma~\ref{eq-(2.1)}, the explicit expression for $\mu_4$ for $my \ge 1$ reads
		\begin{equation}\label{eq-mu41}
			\mu_{4}(mz) = \sum_{n=1}^{\infty}(mz)^{5}(2\pi in)^{3}\sigma(n)\exp(2\pi in mz),
		\end{equation}
		whereas for $my < 1$, Lemma~\ref{eq-2.12} gives
		\begin{equation}\label{eq-mu42}
			\mu_{4}(mz)=\sum_{n=1}^{\infty}P_4\left(\frac{2\pi in}{mz}\right)\sigma(n)\exp\left(-\frac{2\pi in}{mz}\right)+1-\frac{3mz}{2\pi i},   
		\end{equation}
		where $P_4(u)=u^{3}-12u^{2}+36u-24$ and $|P_4(u)|\le |u^{3}|+12|u^{2}|+36|u|+24$.
		\newline
		For $|x| < \frac{y}{3}$, the key exponential and polynomial factors satisfy
		\begin{equation}\label{eq-mu43}
			\begin{aligned}
				& \left|\left(\frac{2\pi in}{mz}\right)^{k}\right| <\frac{(2\pi n)^{k}}{(my)^{k}} \quad \text{and}\quad \left|\exp\left(-\frac{2\pi in}{mz}\right)\right|\le \exp\left(-\frac{9\pi n}{5my}\right).
			\end{aligned}
		\end{equation}  
		We first consider the range $2sty < 1$. From \eqref{eq-mu42}, we have
		\begin{align*}
			|\ell(s,t,z)|\leq 3(s-1)(t-1)+\sum_{j=1}^4|B_j|, 
		\end{align*}
		where
		\begin{align*}
			& B_1=4st\sum_{n=1}^{\infty}P_4\left(\frac{2\pi in}{z}\right)\sigma(n)\exp\left(-\frac{2\pi in}{z}\right)-st\sum_{n=1}^{\infty}P_4\left(\frac{2\pi in}{2z}\right)\sigma(n)\exp\left(-\frac{2\pi in}{2z}\right), \\ 
			& B_2=4\sum_{n=1}^{\infty}P_4\left(\frac{2\pi in}{stz}\right)\sigma(n)\exp\left(-\frac{2\pi in}{stz}\right)-\sum_{n=1}^{\infty}P_4\left(\frac{2\pi in}{2stz}\right)\sigma(n)\exp\left(-\frac{2\pi in}{2stz}\right),\\
			& B_3=t\sum_{n=1}^{\infty}P_4\left(\frac{2\pi in}{2sz}\right)\sigma(n)\exp\left(-\frac{2\pi in}{2sz}\right)-4t\sum_{n=1}^{\infty}P_4\left(\frac{2\pi in}{sz}\right)\sigma(n)\exp\left(-\frac{2\pi in}{sz}\right)\quad \text{and}\\
			& B_4=s\sum_{n=1}^{\infty}P_4\left(\frac{2\pi in}{2tz}\right)\sigma(n)\exp\left(-\frac{2\pi in}{2tz}\right)-4s\sum_{n=1}^{\infty}P_4\left(\frac{2\pi in}{tz}\right)\sigma(n)\exp\left(-\frac{2\pi in}{tz}\right).
		\end{align*}
		Under the conditions $0 < y \le \frac{1}{10}$, $2sty < 1$, and $s > t \ge 10$, applying the estimates in \eqref{eq-mu43} directly yields $\sum_{j=1}^4 |B_j| < 3(s-1)(t-1)$. Consequently,
		\begin{align*}
			|\ell(s,t,z)| < 6(s-1)(t-1).
		\end{align*}
		Next, we address the regime $ty \ge 1$. Equations \eqref{eq-mu41} and \eqref{eq-mu42} imply
		\begin{align*}
			|\ell(s,t,z)| \le \left| 3st - \frac{3st z}{\pi i} \right| + |B_1| + |B_5| + |B_6| + |B_7|,
		\end{align*}
		where $B_1$ is as above, and the remaining terms are
		\begin{align*}
			B_5 &= \sum_{n=1}^{\infty} (stz)^5 (2\pi in)^3 \sigma(n) \left[ 4 \exp({2\pi instz}) - 32 \exp({4\pi instz}) \right], \\
			B_6 &= t \sum_{n=1}^{\infty} (sz)^5 (2\pi in)^3 \sigma(n) \left[ 32 \exp({4\pi insz}) - 4 \exp({2\pi insz}) \right], \\
			B_7 &= s \sum_{n=1}^{\infty} (tz)^5 (2\pi in)^3 \sigma(n) \left[ 32 \exp({4\pi intz}) - 4 \exp({2\pi intz}) \right].
		\end{align*}
		Given $sy > ty \ge 1$, $0 < y \le 1/10$, and $s > t \ge 10$, the exponentially small factors in $B_j$ yield $|B_1| + |B_5| + |B_6| + |B_7| < 2st$. Combining this with $\left| 3st - \frac{3st z}{\pi i} \right| < 4st$ completes the proof of 
		$$|\ell(s,t,z)| < 6st.$$
		Finally, we consider the case $sy\geq1$ and $2ty<1$. Proceeding as in the preceding two cases, and using the corresponding estimates for the error terms, we obtain
		\begin{align*}
			|\ell(s,t,z)|< 6s(t-1).  
		\end{align*}
		Collecting the estimates from the cases considered above and using Lemma~\ref{lemma-mu2}, we deduce
		\begin{align*}
			\left|\frac{\ell(s,t,z)}{b_2(s,t,y)}\right|<36.
		\end{align*}  
	\end{proof}
	\section{Proof of Main Results}\label{sec-4}
	In this section, we prove all results presented in the introduction. Let $\vartheta \in \mathbb{C}$ such that $|\vartheta|\le1.$ Before establishing our main results, we first prove a preliminary proposition required for the derivation of the asymptotic formula 
	for $\bar{p}_{s,t}(N)$. In the proof, the value of $\vartheta$ may depend on the relevant parameters and may vary from one occurrence to another. In Theorem \ref{thm-1.1}$(i)$, we prove that $y$ satisfies (\ref{eq-solmain}). Hence, we can take such a $y$ in the following proposition without affecting its generality.
	\begin{proposition}\label{thm-main}
		Recall $b_k(s,t,y)$ from \eqref{eq-bksty1}, and choose $y$ such that
		\begin{equation*}
			\left|\frac{-b_1(s,t,y)
			}{2sty^{2}}- N  \right|<\frac{2}{25y}   
		\end{equation*}
		and suppose that $y\le \frac{1}{1000}$ and $s>t\ge 50$ are relatively prime integers. Then,
		\begin{align*}
			\bar{p}_{s,t}(N)= \frac{y^{\frac{3}{2}}\sqrt{2st}\,\exp\left(2\pi Ny\right)H(iy,s,t)}{\sqrt{b_2(s,t,y)}}
			\left(1+O\left( \frac{sty}{b_2(s,t,y)}\right)\right).
		\end{align*}
	\end{proposition}
	\begin{proof}
		Recall the integral representation for $\bar{p}_{s,t}(N)$ given in \eqref{eq-qNt}. 
		Applying Lemma~\ref{pro-error}, we obtain
		\begin{align}\label{eq-qNt1}
			\bar{p}_{s,t}(N)&=\exp(2\pi Ny)H(iy,s,t)\int_{-y/3}^{y/3}\exp\left(-2\pi i Nx+2\pi i\frac{1}{2\pi i}\log\frac{H(z,s,t)}{H(iy,s,t)}\right)dx\\
			\notag
			&+\exp(2\pi Ny)H(iy,s,t)\left(\vartheta 1.5\exp\left(-\frac{\pi }{600y}\right)\right).
		\end{align}
		The Taylor expansion of \eqref{eq-tay1} yields
		\begin{align*}
			\frac{1}{2\pi i}\log \frac{H(z,s,t)}{H(iy,s,t)}&=x\frac{b_1(s,t,y)}{2st(iy)^{2}}+\frac{x^{2}}{2!}\frac{b_2(s,t,y)}{2st(iy)^{3}}+\frac{x^{3}}{3!}\frac{b_3(s,t,y)}{2st(iy)^{4}}+\frac{x^{4}}{4!}\frac{\ell(s,t,z')}{2st(z^{\prime})^{5}},
		\end{align*}
		where $\ell(s,t,z')$ is defined in Lemma \ref{lemma-mu4}, and $z' = x' + iy$ for some $x' \in (0,x)$. Combining Lemma~\ref{lemma-mu3} and Lemma~\ref{lemma-mu4}, we deduce that
		\begin{align*}
			\frac{1}{2\pi i}\log \frac{H(z,s,t)}{H(iy,s,t)}&=x\frac{b_1(s,t,y)}{2st(iy)^{2}}
			+\frac{x^{2}}{2}\frac{b_2(s,t,y)}{2st(iy)^{3}}\left(1+2i\epsilon_2\frac{x}{y}+\vartheta 3\frac{x^{2}}{y^{2}}\right),
		\end{align*}
		where $\epsilon_2 \in (-1,1)$ and $|\vartheta|\le 1.$
		\newline
		Let 
		\begin{align*}
			\alpha=\frac{b_2(s,t,y)}{2sty}\qquad \text{and  }\qquad \beta=y\left(\frac{b_1(s,t,y)}{2st(iy)^{2}}-N\right). 
		\end{align*}
		Since we have already assumed in Proposition \ref{thm-main} that $y\le \frac{1}{1000}$ and $s>t\ge 50$, it follows from Lemma \ref{lemma-mu2} that $\alpha>38$. Moreover, under the first assumption of Proposition \ref{thm-main}, $|\beta|<\frac{2}{25}$. Now the integrand of (\ref{eq-qNt1}) further simplifies to
		\begin{equation*}
			\begin{aligned}
				&\exp\left(-2\pi i Nx+2\pi i\frac{1}{2\pi i}\log\frac{H(z,s,t)}{H(iy,s,t)}\right)\\
				&=\exp\left(\frac{2\pi ix}{y}y\left(\frac{b_1(s,t,y)}{2st(iy)^{2}}-N\right)-\pi\frac{x^{2}}{y^{2}}\frac{b_2(s,t,y)}{2sty}\left(1+2i\epsilon_2\frac{x}{y}+\vartheta 3\frac{x^{2}}{y^{2}}\right)\right)\\
				&=\exp\left(\frac{2\pi i\beta x}{y}\right)\exp\left(-\pi\alpha\frac{x^{2}}{y^{2}}\left(1+2i\epsilon_2\frac{x}{y}+\vartheta 3\frac{x^{2}}{y^{2}}\right)\right).
			\end{aligned}  
		\end{equation*}
		Applying the change of variable $w = \frac{x}{y}$ and employing Lemma~4.1 of \cite{tyler} to the integral in (\ref{eq-qNt1}), we obtain
		\begin{align*}
			\int_{-1/3}^{1/3}\exp\left(2\pi i\beta w\right)\exp\left(-\pi\alpha w^{2}\left(1+2i\epsilon_2 w+\vartheta 3w^{2}\right)\right)y\,dw=\frac{y}{\sqrt{\alpha}}\left(1+\vartheta\frac{3.45}{\alpha}\right).
		\end{align*}
		Combining the above equation with \eqref{eq-qNt1}, we see that
		\begin{align*}
			\bar{p}_{s,t}(N) 
			&=\exp(2\pi Ny)H(iy,s,t)\frac{y}{\sqrt{\alpha}}\left(1+\vartheta\frac{3.45}{\alpha}+\vartheta\frac{\sqrt{\alpha}}{y} 1.5\exp\left(-\frac{\pi }{600y}\right)\right). 
		\end{align*}
		Since $y \le \frac{1}{1000}$ and $s>t\ge 50$, applying the bound $b_2(s,t,y) < \frac{(s+1)(t+1)}{4}$ 
		from Lemma~\ref{lemma-mu2} yields
		\begin{align*}
			\frac{\alpha^{\frac{3}{2}}}{y}\exp\left(-\frac{\pi }{600y}\right)=O(1)\quad \text{with } \alpha = \frac{b_2(s,t,y)}{2sty}.  
		\end{align*}
		Hence, the preceding estimate reduces to
		\begin{align*}
			\bar{p}_{s,t}(N)&=\exp(2\pi Ny)H(iy,s,t)\frac{y}{\sqrt{\alpha}}\left(1+O\left(\frac{1}{\alpha}\right)\right)  \\
			&=\frac{y^{\frac{3}{2}}\sqrt{2st}\exp\left(2\pi Ny\right)H(iy,s,t)}{\sqrt{b_2(s,t,y)}}
			\left(1+O\left( \frac{sty}{b_2(s,t,y)}\right)\right).  
		\end{align*}
		This completes the proof.  
	\end{proof}
	We now prove our main results. Throughout the proofs, we let $m$ denote an arbitrary positive integer. We first recall $H(z,s,t)$ from (\ref{eq-Hzst}) and $b_k(s,t,y)$ from (\ref{eq-bksty}):
	\begin{align*}
		H(z,s,t)&=\frac{\eta(2z)\eta(2stz)\eta(sz)^{2}\eta(tz)^{2}}{\eta(z)^{2}\eta(stz)^{2}\eta(2sz)\eta(2tz)} \quad\text{and}\\  b_k(s,t,y)&= 4st\mu_k(iy)+4\mu_k(stiy)+t\mu_k(2siy)+s\mu_k(2tiy)\\
		& -st\mu_{k}(2iy)-\mu_k(2stiy)-4t\mu_k(siy)-4s\mu_k(tiy).
	\end{align*}
	In part $(i)$ of the following proof, we prove the uniqueness of the saddle point $y$, while in part $(ii)$, we establish a general asymptotic formula for $\bar{p}_{s,t}(N)$.
	\begin{proof}[\textbf{\boldmath Proof of Theorem \ref{thm-1.1}$(i)$}] 
		We wish to determine the saddle point $y$ by solving the equation $\frac{d}{dz}\log H(z,s,t)=2\pi i N$ at $z=iy$. This gives
		\begin{align*}
			2\pi iN 
			&=\frac{d}{dz}\log \eta(2z)+\frac{d}{dz}\log \eta(2stz)+2\frac{d}{dz}\log \eta(sz)+2\frac{d}{dz}\log \eta(tz)\\
			&-2\frac{d}{dz}\log \eta(z)-2\frac{d}{dz}\log \eta(stz)-\frac{d}{dz}\log \eta(2sz)-\frac{d}{dz}\log \eta(2tz).    
		\end{align*}
		Using $\frac{d}{dz}\log \eta(mz)=-\frac{2\pi i}{m z^{2}}\mu_1(mz)$ from definition of $\mu_k$ and substituting $z=iy$, we obtain
		\begin{equation*}
			\frac{st\mu_{1}(2iy)+\mu_1(2stiy)+4t\mu_1(siy)+4s\mu_1(tiy)-4st\mu_1(iy)-4\mu_1(stiy)-t\mu_1(2siy)-s\mu_1(2tiy)}{2sty^{2}}=N.    
		\end{equation*}
		Thus, $y$ is a solution of $\frac{-b_1(s,t,y)}{2sty^{2}}=N.$
		Next, we show that the solution $y>0$ is unique.\\ For $my \ge 1$, Lemma~\ref{eq-(2.1)} yields
		\begin{equation}\label{eq-pf11}
			\mu_{1}(miy)=\frac{m^{2}y^{2}}{24}-\sum_{n=1}^{\infty}m^{2}y^{2}\sigma(n)\exp(-2\pi mny), 
		\end{equation} 
		and for $my<1$, Lemma \ref{eq-2.12} gives
		\begin{equation}\label{eq-pf12}
			\mu_{1}(miy)=-\frac{1}{24}+\frac{my}{4\pi}+\sum_{n=1}^{\infty}\sigma(n)\exp\left(-\frac{2\pi n}{my}\right).
		\end{equation}
		From the above two equations, it follows that
		\begin{align*}
			\lim_{y\to \infty}\frac{-b_1(s,t,y)}{2sty^{2}}=0 \quad \text{and   }\lim_{y\to 0^{+}}\frac{-b_1(s,t,y)}{2sty^{2}}=\infty.  
		\end{align*}
		Hence, for any positive integer $N$, there exists $y > 0$ such that
		\begin{align*}
			\frac{-b_1(s,t,y)}{2sty^{2}}=N.  
		\end{align*}
		From the explicit formulas for $\mu_k(z)$ given in Lemma~\ref{eq-(2.1)} and Lemma~\ref{eq-2.12}, we have\footnote{The proof proceeds by considering the cases $y\ge 1$ and $y<1$, and follows directly from Lemma~\ref{eq-(2.1)} and Lemma~\ref{eq-2.12}.} $$\mu_{2}(z)=-2\mu_{1}(z)+z\mu_{1}^{\prime}(z).$$
		Invoking Lemma~\ref{lemma-mu2}, we deduce
		\begin{align*}
			\frac{d}{dy}\left(\frac{-b_1(s,t,y)}{2sty^{2}}\right)=\frac{-b_2(s,t,y)}{2sty^{3}}<0.   
		\end{align*}
		Hence, the solution $y>0$ is unique. 
	\end{proof}
	\begin{proof}[\textbf{\boldmath Proof of Theorem \ref{thm-1.1}$(ii)$}]
		Since in (\ref{eq-solmain}), we proved that $y$ is a solution of $\frac{-b_1(s,t,y)}{2sty^{2}}-N=0$, so it is obvious that $\left|\frac{-b_1(s,t,y)}{2sty^{2}}-N\right|\ll\frac{1}{y}$. Consequently, by invoking Proposition~\ref{thm-main} and 
		the lower bound for $b_2(s,t,y)$ from Lemma~\ref{lemma-mu2}, we deduce that
		\begin{align*}
			\bar{p}_{s,t}(N)&=\frac{y^{\frac{3}{2}}\sqrt{2st}\exp\left(2\pi Ny\right)H(iy,s,t)}{\sqrt{b_2(s,t,y)}}
			\left(1+O(y)\right).  
		\end{align*} 
	\end{proof}
	We now derive explicit asymptotic bounds for $\bar{p}_{s,t}(N)$ in terms of $N$, $s$, and $t$ over different ranges of $s$ and $t$.
	\begin{proof}[\textbf{\boldmath Proof of Theorem \ref{thm-1.2}}]
		The proof is carried out under the assumption $2sty<1$. By Lemma~\ref{lemma-2.2}, for any $my<1$ and $1<\nu<1.01$, we have
		\begin{equation}\label{eq-etaS}
			\eta(miy) = (my)^{-\frac{1}{2}} \exp\left( -\frac{\pi}{12my} - \nu e^{-\frac{2\pi}{my}} \right).
		\end{equation}
		Applying the above approximation, we find
		\begin{align*}
			H(iy,s,t)=\exp\left(\frac{\pi(s-1)(t-1)}{8sty}\right)E_1(s,t,y)E_2(s,t,y),
		\end{align*}
		where $1 < \nu_j < 1.01$ for each $1 \le j \le 8$, and
		\begin{align*}
			E_1(s,t,y)&=\exp\left(-\nu_1e^{-\frac{2\pi}{2y}}-\nu_2e^{-\frac{2\pi}{2sty}}-2\nu_3e^{-\frac{2\pi}{sy}}-2\nu_4e^{-\frac{2\pi}{ty}}\right)\quad\text{and}\\
			E_2(s,t,y)&=\exp\left(2\nu_5e^{-\frac{2\pi}{y}}+2\nu_6e^{-\frac{2\pi}{sty}}+\nu_7e^{-\frac{2\pi}{2sy}}+\nu_8e^{-\frac{2\pi}{2ty}}\right).
		\end{align*}
		Substituting $H(iy,s,t)$ in Theorem~\ref{thm-1.1}$(ii)$ yields
		\begin{align}\label{eq-qNt4}
			\bar{p}_{s,t}(N)&=\frac{y^{\frac{3}{2}}\sqrt{2st}\,\exp\left(2\pi Ny+\frac{\pi (s-1)(t-1)}{8sty}\right)E_1(s,t,y)E_2(s,t,y)}{\sqrt{b_2(s,t,y)}}
			\left(1+O(y)\right).  
		\end{align}
		Next, we solve for $y$ explicitly using Theorem~\ref{thm-1.1}$(i)$. 
		Inserting $\mu_1$ from \eqref{eq-pf12} in Theorem~\ref{thm-1.1}$(i)$, we obtain
		\begin{align*}
			2sty^{2}N=\frac{(s-1)(t-1)}{8}+K_1+K_2+K_3+K_4,
		\end{align*}
		where
		\begin{align*}
			&K_1=  st\sum_{n=1}^{\infty}\sigma(n)\left[\exp\left(-\frac{2\pi n}{2y}\right)-4 \exp\left(-\frac{2\pi n}{y}\right)\right],\\
			&K_2= \sum_{n=1}^{\infty}\sigma(n)\left[\exp\left(-\frac{2\pi n}{2sty}\right)-4\exp\left(-\frac{2\pi n}{sty}\right)\right], \\
			&K_3=t\sum_{n=1}^{\infty}\sigma(n)\left[4\exp\left(-\frac{2\pi n}{sy}\right)-\exp\left(-\frac{2\pi n}{2sy}\right)\right]\,\,\text{and  } \\
			&K_4=s\sum_{n=1}^{\infty}\sigma(n)\left[4\exp\left(-\frac{2\pi n}{ty}\right)-\exp\left(-\frac{2\pi n}{2ty}\right)\right].
		\end{align*}
		Since we have already assumed that $2sty<1$, it follows from the above that $K_1+K_2+K_3+K_4=O(K_2).$ Substituting the crude approximation $y\approx \sqrt{\frac{(s-1)(t-1)}{16stN}}$ in $K_2$\footnote{Note $\frac{1}{2st}K_2=\frac{1}{2st}\left(\sum_{n=1}^{\infty}\sigma(n)\left[\exp\left(-\frac{2\pi n}{2sty}\right)-4\exp\left(-\frac{2\pi n}{sty}\right)\right]\right)=O\left(N^{-1}\right)$ for $2st \le \frac{2\pi}{\left(\frac{1}{2} + \epsilon\right) \log N} \sqrt{\frac{16stN}{(s-1)(t-1)}}$. } gives
		\begin{align}\label{eq-solutiony}
			y=\sqrt{\frac{(s-1)(t-1)}{16stN}}+O\left(N^{-\frac{3}{2}}\right).  
		\end{align}
		We can also compute 
		\begin{align}\label{eq-yinvese}
			y^{-1}=\sqrt{\frac{16stN}{(s-1)(t-1)}}+O\left(N^{-\frac{1}{2}}\right).    
		\end{align}
		Using the value of $y$ and $y^{-1}$, we obtain the following estimates
		\begin{align}\label{eq-estimate1}
			\notag
			&2\pi Ny+\frac{\pi (s-1)(t-1)}{8sty}=\pi\sqrt{\frac{N(s-1)(t-1)}{st}}+O\left(N^{-\frac{1}{2}}\right),\\
			&E_1(s,t,y)=1+O\left(N^{-\frac{1}{2}}\right)\quad \text{and}\quad E_2(s,t,y)=1+O\left(N^{-\frac{1}{2}}\right).
		\end{align}
		\begin{claim}\label{claim-3}	
			For any fixed $\epsilon > 0$, if $s$ and $t$ are coprime integers satisfying $2sty < 1$ and $2st \le \frac{2\pi}{\left(\frac{1}{2} + \epsilon\right) \log N} \sqrt{\frac{16stN}{(s-1)(t-1)}}$, then
			\begin{equation*}
				b_2(s,t,y)=\frac{(s-1)(t-1)}{4}\left(1+O\left(N^{-\frac{1}{2}}\right)\right).  
			\end{equation*}   
		\end{claim}
		\begin{proof}[\textbf{\boldmath Proof of Claim \ref{claim-3}}]
			By Lemma \ref{eq-2.12}, for $my<1$, we have
			\begin{align}\label{eq-mu22}
				\mu_2(miy)=\sum_{n=1}^{\infty}\left(\frac{2\pi n}{my}-2\right)\sigma(n)\exp\left(-\frac{2\pi n}{my}\right)+\frac{1}{12}-\frac{my}{4\pi}.
			\end{align} 
			Combining the above expression for $\mu_2$ in the specified ranges of $s$ and $t$ with \eqref{eq-yinvese}, we obtain
			\begin{equation*}
				b_2(s,t,y)=\frac{(s-1)(t-1)}{4}\left(1+O\left(N^{-\frac{1}{2}}\right)\right).  
			\end{equation*}   
		\end{proof}
		Inserting $y$ from \eqref{eq-solutiony} and the estimates of \eqref{eq-estimate1} 
		and Claim~\ref{claim-3} in \eqref{eq-qNt4}, we obtain
		\begin{align*}
			\bar{p}_{s,t}(N)=\frac{1}{2\sqrt{2}\,N^{\frac{3}{4}}}\left(\frac{(s-1)(t-1)}{st}\right)^{\frac{1}{4}}\exp\left(\pi\sqrt{\frac{N(s-1)(t-1)}{st}}\right)\left(1+O\left(N^{-\frac{1}{2}}\right)\right).   
		\end{align*}
		This completes the proof.	
	\end{proof}
	
	We are now ready to prove Theorem~\ref{thm-1.3}, under the assumption that $ty \geq 1$.
	\begin{proof}[\textbf{\boldmath Proof of Theorem \ref{thm-1.3}}]
		For any $my \geq 1$ and $1<\delta<1.01$, Lemma~\ref{lemma-2.1} gives 
		\begin{align}\label{eq-etaL}
			\eta(miy)=\exp\left(-\frac{\pi my}{12}-\delta e^{-2\pi my}\right).  
		\end{align}
		Applying the above equation to $\eta(2stiy)$, $\eta(stiy)$, $\eta(2siy)$, $\eta(siy)$, $\eta(2tiy)$, and $\eta(tiy)$, and applying (\ref{eq-etaS}) to $\eta(iy)$ and $\eta(2iy)$, we obtain
		\begin{align*}
			H(iy,s,t)=\frac{\sqrt{y}}{\sqrt{2}}\exp\left(\frac{\pi}{8y}\right)E_3(s,t,y)E_4(s,t,y),
		\end{align*}
		where $1<\delta_j, \nu_j<1.01$,
		\begin{align*}
			E_3(s,t,y)&=\exp\left(\delta_1e^{-4\pi sy}-2\delta_2e^{-2\pi sy}+\delta_3e^{-4\pi ty}-2\delta_4e^{-2\pi ty}\right)
			\quad\text{and}\\ 
			E_4(s,t,y)&=\exp\left(2\nu_9e^{-\frac{2\pi }{y}}-\nu_{10}e^{-\frac{2\pi }{2y}} +2\delta_5e^{-2\pi sty}-\delta_6e^{-4\pi sty}\right).
		\end{align*}
		Inserting the above $H(iy,s,t)$ in Theorem~\ref{thm-1.1}$(ii)$, we arrive at
		\begin{align}\label{eq-pNst2}
			\bar{p}_{s,t}(N)=\frac{y^{2}\sqrt{st}\,\exp\left(2\pi Ny+\frac{\pi }{8y}\right)E_3(s,t,y)E_4(s,t,y)}{\sqrt{b_2(s,t,y)}}\left(1+O(y)\right).   
		\end{align}
		We now determine $y$ from Theorem~\ref{thm-1.1}$(i)$ under the assumption $ty \geq 1$. 
		Applying \eqref{eq-pf11} to $\mu_1(2stiy)$, $\mu_1(stiy)$ $\mu_1(2siy)$, $\mu_1(siy)$, $\mu_1(2tiy)$ and $\mu_1(tiy)$, 
		along with \eqref{eq-pf12} to $\mu_1(2iy)$ and $\mu_1(iy)$, we deduce
		\begin{align}\label{eq-sdy2}
			y^{2}N+\frac{y}{4\pi}-\frac{1}{16}+\frac{K_5+K_6+K_7+K_8}{2st}=0, \end{align}
		where 
		\begin{align*}
			K_5&= st\sum_{n=1}^{\infty}\sigma(n)\left[-\exp\left(-\frac{2\pi n}{2y}\right)+4\exp\left(-\frac{2\pi n}{y}\right)\right],\\
			K_6&=\sum_{n=1}^{\infty}(2sty)^{2}\sigma(n)\left[\exp(-4\pi nsty)-\exp(-2\pi nsty)\right],\\
			K_7&=4t\sum_{n=1}^{\infty}(sy)^{2}\sigma(n)\left[\exp(-2\pi nsy)-\exp(-4\pi nsy)\right]\quad\text{and}\\
			K_8&=4s\sum_{n=1}^{\infty}(ty)^{2}\sigma(n)\left[\exp(-2\pi nty)-\exp(-4\pi nty)\right].
		\end{align*}
		Treating \eqref{eq-sdy2} as a quadratic equation in $y$, we obtain
		\begin{align*}
			y=-\frac{1}{8\pi N}+\frac{1}{4\sqrt{N}}\sqrt{1-\frac{8(K_5+K_6+K_7+K_8)}{st}+\frac{1}{4\pi^{2}N}}.  
		\end{align*}
		We bound $K_5$, $K_6$, $K_7$ and $K_8$ using the crude approximation $y \approx \frac{1}{4\sqrt{N}}$.\footnote{Since $s > t > 4\sqrt{N}$ and $y = \frac{1}{4\sqrt{N}}$, we get $sty \ge 4\sqrt{N}$ and $\frac{1}{2st}(K_5+K_6) = O\left(N^{-2}\right)$. For $ty \ge 1$, we have $ \frac{1}{2st}(K_7+K_8) \le C_2 N^{-1/2}$ for some constant $C_2$ independent of $N$.} Hence,
		\begin{align}\label{eq-sol2}
			y=\frac{1}{4\sqrt{N}}+\frac{C_1}{N}+\frac{C_3(N,s,t)}{N}+O\left(N^{-\frac{3}{2}}\right),   
		\end{align}
		where $C_3(N,s,t)$ depends on $C_2$ and the constant $C_1$ is independent of $N$, $s$ and $t$.
		\newline
		We may also compute $y^{-1}$ as
		\begin{align}\label{eq-sol2yin}
			y^{-1}=  4\sqrt{N}\left(1-\frac{4C_1}{\sqrt{N}}-\frac{4C_3(N,s,t)}{\sqrt{N}}+O\left(N^{-1}\right)\right). 
		\end{align}
		Using \eqref{eq-sol2} and \eqref{eq-sol2yin}, and setting $E_3(N,s,t) := E_3(s,t,y)$, we derive the following estimates:
		\begin{align}\label{eq-esti2}
			&2\pi Ny+\frac{\pi }{8y}=\pi \sqrt{N}+O\left(N^{-\frac{1}{2}}\right), \quad E_4(s,t,y)=\left(1+O\left(N^{-\frac{1}{2}}\right)\right)\quad \text{and}\\ 
			\notag
			E_3(N,s,t)=&\exp\left(\delta_1\exp\left(-\frac{\pi s}{\sqrt{N}}\left(1+O\left(N^{-\frac{1}{2}}\right)\right)\right)-2\delta_2\exp\left(-\frac{\pi s}{2\sqrt{N}}\left(1+O\left(N^{-\frac{1}{2}}\right)\right)\right)\right)\\
			\notag
			\times & \exp\left(\delta_3\exp\left(-\frac{\pi t}{\sqrt{N}}\left(1+O\left(N^{-\frac{1}{2}}\right)\right)\right)-2\delta_4\exp\left(-\frac{\pi t}{2\sqrt{N}}\left(1+O\left(N^{-\frac{1}{2}}\right)\right)\right)\right).
		\end{align}
		\begin{claim}\label{claim-2}
			If $s$ and $t$ are coprime integers with $s > t > 4\sqrt{N}$ and $ty \ge 1$, then
			\begin{align*}
				b_2(s,t,y)= \frac{st}{4}\left(1+O\left(N^{-\frac{1}{2}}\right)\right).  
			\end{align*}
		\end{claim}
		\begin{proof}[\textbf{\boldmath Proof of Claim \ref{claim-2}}]
			For any $my\ge 1$, Lemma \ref{eq-(2.1)} gives
			\begin{align}\label{eq-mu221}
				\mu_2(miy)=\sum_{n=1}^{\infty}(my)^{3}(2\pi n)\sigma(n)\exp(-2\pi mny).
			\end{align}
			Using the above equation for $\mu_2(2stiy)$, $\mu_2(stiy)$, $\mu_2(2siy)$, $\mu_2(siy)$, $\mu_2(2tiy)$, and $\mu_2(tiy)$, and applying \eqref{eq-mu22} to $\mu_2(2iy)$ and $\mu_2(iy)$, we obtain
			\begin{align*}
				b_2(s,t,y)=\frac{st}{4}\left(1+O\left(N^{-\frac{1}{2}}\right)\right).
			\end{align*}
		\end{proof}
		Substituting $y = \frac{1}{4\sqrt{N}}\left(1+O\left(N^{-\frac{1}{2}}\right)\right)$ 
		from \eqref{eq-sol2}, Claim~\ref{claim-2}, and the estimates from \eqref{eq-esti2} 
		in \eqref{eq-pNst2}, and applying the formula for $\bar{p}(N)$ from \eqref{eq-pN}, 
		we derive the following simplified expression for $\bar{p}_{s,t}(N)$:
		\begin{align*}
			\bar{p}_{s,t}(N)=&\bar{p}(N)E_3(N,s,t)\left(1+O\left(N^{-\frac{1}{2}}\right)\right). 
		\end{align*}
		This completes the proof.
	\end{proof}
	Throughout the following proof, we assume that $2ty<1$ and $sy\ge 1$.
	\begin{proof}[\textbf{\boldmath Proof of Theorem~\ref{thm-1.4}}]
		Applying \eqref{eq-etaS} to $\eta(2iy)$, $\eta(iy)$, $\eta(2tiy)$, and
		$\eta(tiy)$, and \eqref{eq-etaL} to $\eta(2stiy)$,
		$\eta(stiy)$, $\eta(2siy)$ and $\eta(siy)$, we obtain the following expression for
		$H(iy,s,t)$:
		\begin{align*}
			H(iy,s,t)= \frac{1}{\sqrt{t}}\exp\left(\frac{\pi (t-1)}{8ty}\right)E_5(s,t,y)E_6(s,t,y), 
		\end{align*}
		where $1<\delta_j, \nu_j<1.01$,
		\begin{align*}
			E_5(s,t,y)&=\exp\left(2\delta_7e^{-2\pi sty}-\delta_8e^{-4\pi sty}+\delta_9e^{-4\pi sy}-2\delta_{10}e^{-2\pi sy}\right)\quad\text{and}\\
			E_6(s,t,y)&=\exp\left(2\nu_{11}e^{-\frac{2\pi}{y}}-\nu_{12}e^{-\frac{2\pi}{2y}}+\nu_{13}e^{-\frac{2\pi}{2ty}}-2\nu_{14}e^{-\frac{2\pi}{ty}}\right).
		\end{align*}
		Inserting the above $H(iy,s,t)$ in Theorem~\ref{thm-1.1}, we derive
		\begin{align}\label{eq-barp3}
			\bar{p}_{s,t}(N)=\frac{y^{\frac{3}{2}}\sqrt{2s}\exp\left(2\pi Ny+\frac{\pi(t-1)}{8ty}\right)E_5(s,t,y)E_6(s,t,y)}{\sqrt{b_2(s,t,y)}}\left(1+O(y)\right).   
		\end{align}
		Similarly, following the same procedure as in the proofs of
		Theorems~\ref{thm-1.2} and \ref{thm-1.3}, we determine $y$ from Theorem~\ref{thm-1.1} in the given ranges of $s$ and $t$, and obtain
		\begin{align*}
			y= \frac{1}{4}\sqrt{\frac{t-1}{tN}}+\frac{C_4(N,s,t)}{N}+O\left(N^{-\frac{3}{2}}\right), 
		\end{align*}
		where $|C_4(N,s,t)|\leq C_5$ for some constant $C_5$ independent of
		$N$, $s$, and $t$.\\
		For the above $y$ and setting $E_5(N,s,t):=E_5(s,t,y)$, we obtain the following estimates in the specified range of $t$:
		\begin{align*}
			&2\pi Ny+ +\frac{\pi(t-1)}{8ty}=\pi \sqrt{N\left(1-\frac{1}{t}\right)}+O\left(N^{-\frac{1}{2}}\right), \quad E_6(s,t,y)=1+O\left(N^{-\frac{1}{2}}\right)\quad\text{and}\\
			&E_5(N,s,t)\\
			&=\exp\left(2\delta_7\exp\left(-\frac{\pi s}{2}\sqrt{\frac{t(t-1)}{N}}\left(1+O\left(N^{-\frac{1}{2}}\right)\right)\right)-\delta_8\exp\left(-\pi s\sqrt{\frac{t(t-1)}{N}}\left(1+O\left(N^{-\frac{1}{2}}\right)\right)\right)\right)\\
			&\times \exp\left(\delta_9\exp\left(-\pi s\sqrt{\frac{(t-1)}{tN}}\left(1+O\left(N^{-\frac{1}{2}}\right)\right)\right)-2\delta_{10}\exp\left(-\frac{\pi s}{2}\sqrt{\frac{(t-1)}{tN}}\left(1+O\left(N^{-\frac{1}{2}}\right)\right)\right)\right).
		\end{align*}
		\begin{claim}
			For the above value of $y$, if $s>4\sqrt{N}$ and $2t\le \frac{8\pi}{\left(\frac{1}{2}+\epsilon_1\right)\log N}\sqrt{\frac{tN}{t-1}}$ for some fixed $\epsilon_1>0$, then
			\begin{align*}
				b_2(s,t,y)=\frac{s(t-1)}{4}\left(1+O\left(N^{-\frac{1}{2}}\right)\right).    
			\end{align*}
		\end{claim}
		\begin{proof}
			The proof follows from \eqref{eq-mu22} and \eqref{eq-mu221} in the specified range of $s$ and $t$.
		\end{proof}
		Inserting $y$, the corresponding estimates, and the above claim in \eqref{eq-barp3}, we obtain
		\begin{align*}
			\bar{p}_{s,t}(N)=\frac{1}{2\sqrt{2t}\,N^{\frac{3}{4}}}\left(1-\frac{1}{t}\right)^{\frac{1}{4}}\exp\left(\pi \sqrt{N\left(1-\frac{1}{t}\right)}\right)E_5(N,s,t)\left(1+O\left(N^{-\frac{1}{2}}\right)\right). 
		\end{align*}
	\end{proof}
	We are now in a position to prove the higher-order Turán inequalities for $\bar{p}_{s,t}(N)$.
	\begin{proof}[\textbf{\boldmath Proof of Theorem~\ref{thm-turan}}]
		By Theorem~\ref{thm-1.2}, we can write
		\begin{align*}
			\bar{p}_{s,t}(N)\sim \frac{1}{2\sqrt{2}\,N^{\frac{3}{4}}}\left(\frac{(s-1)(t-1)}{st}\right)^{\frac{1}{4}}\exp\left(\pi\sqrt{\frac{N(s-1)(t-1)}{st}}\right).
		\end{align*}
		Hence, for any $r\ge 0$, we have
		\begin{align*}
			\frac{\bar{p}_{s,t}(N+r)}{\bar{p}_{s,t}(N)}\sim \left(\frac{N}{N+r}\right)^{\frac{3}{4}}\exp\left(\pi\sqrt{\frac{(s-1)(t-1)}{st}}\left(\sqrt{N+r}-\sqrt{N}\right)\right).  
		\end{align*}
		Taking the natural logarithm on both sides, we have
		\begin{align}\label{eq-Tu1}
			\log\left(\frac{\bar{p}_{s,t}(N+r)}{\bar{p}_{s,t}(N)}\right)\sim \pi\sqrt{\frac{(s-1)(t-1)}{st}}\left(\sqrt{N+r}-\sqrt{N}\right)-\frac{3}{4}\log \left(1+\frac{r}{N}\right).
		\end{align}
		Using $$
		\log \left(1+\frac{r}{N}\right)=\sum_{j=1}^{\infty}\frac{(-1)^{j+1}r^{j}}{jN^{j}} \quad\text{and}\quad \sqrt{N+r}-\sqrt{N}=\sum_{j=1}^{\infty}\binom{\frac{1}{2}}{j}\frac{r^{j}}{N^{j-\frac{1}{2}}},
		$$ 
		the right-hand side of \eqref{eq-Tu1} simplifies to
		\begin{align*}
			&\sum_{j=1}^{\infty} \left(\pi\sqrt{\frac{(s-1)(t-1)}{st}}\binom{\frac{1}{2}}{j}\frac{1}{N^{j-\frac{1}{2}}}+\frac{3}{4}\frac{(-1)^{j}}{jN^{j}}\right)r^{j}\\
			&=\left(\frac{\pi }{2}\sqrt{\frac{(s-1)(t-1)}{stN}}-\frac{3}{4N}\right)r-\left(\frac{\pi }{8}\sqrt{\frac{(s-1)(t-1)}{stN^{3}}}-\frac{3}{8N^{2}}\right)r^{2}\\
			&+\sum_{j=3}^{\infty}\left(\pi\sqrt{\frac{(s-1)(t-1)}{st}}\binom{\frac{1}{2}}{j}\frac{1}{N^{j-\frac{1}{2}}}+\frac{3}{4}\frac{(-1)^{j}}{jN^{j}}\right)r^{j}.
		\end{align*}
		Comparing this expression with Theorem~\ref{thm_GORZ}, we set
		\begin{align*}
			L(N)&=\frac{\pi }{2}\sqrt{\frac{(s-1)(t-1)}{stN}}-\frac{3}{4N}, \quad\delta(N)=\sqrt{\frac{\pi }{8}\sqrt{\frac{(s-1)(t-1)}{stN^{3}}}-\frac{3}{8N^{2}}},\\
			\text{and}\quad &h_j(N)=\pi\sqrt{\frac{(s-1)(t-1)}{st}}\binom{\frac{1}{2}}{j}\frac{1}{N^{j-\frac{1}{2}}}+\frac{3}{4}\frac{(-1)^{j}}{jN^{j}},\quad\text{for all $j\ge 3$.}
		\end{align*}
		A straightforward calculation shows that, for $3\leq j\leq d$,
		\begin{align*}
			\lim_{N\to\infty}\frac{h_j(N)}{\delta(N)^j}=0,
		\end{align*}
		while, for $j\geq d+1$,
		\begin{align*}
			\lim_{N\to\infty}\frac{h_j(N)}{\delta(N)^d}=0.
		\end{align*}
		Thus, the sequences $\{\bar{p}_{s,t}(N)\}$, $\{L(N)\}$, $\{\delta(N)\}$, and
		$\{h_j(N)\}$ satisfy the hypotheses of Theorem~\ref{thm_GORZ}.
		Consequently, the Jensen polynomials associated with $\bar{p}_{s,t}(N)$ admit
		an asymptotic representation in terms of Hermite polynomials. Since
		Hermite polynomials are hyperbolic, the corresponding Jensen polynomials
		are also hyperbolic for all sufficiently large $N$.
	\end{proof}
	\subsection*{AI Declaration}
	The author declares that no artificial intelligence tools or AI-assisted technologies were
	used in the preparation, writing, or mathematical analysis of this manuscript.
	
\end{document}